\documentclass[11pt]{amsart}

\usepackage[utf8]{inputenc}
\usepackage{amsmath}
\usepackage{amsfonts}
\usepackage{amsthm}
\usepackage{mathtools}
\usepackage{enumitem}
\usepackage{mathrsfs}
\usepackage{datetime}
\usepackage[hidelinks]{hyperref}
\usepackage{bookmark}

\newtheorem{theorem}{Theorem}[section]
\newtheorem{proposition}[theorem]{Proposition}
\newtheorem{lemma}[theorem]{Lemma}
\newtheorem{corollary}[theorem]{Corollary}
\theoremstyle{definition}
\newtheorem{definition}[theorem]{Definition}
\newtheorem{example}[theorem]{Example}
\newtheorem{remark}[theorem]{Remark}

\numberwithin{equation}{section}

\DeclareMathOperator{\Span}{Span}
\newcommand{\rspan}{\Span_{\mathbb{R}}}
\newcommand{\Q}{\mathbb{Q}}
\newcommand{\R}{\mathbb{R}}
\newcommand{\C}{\mathbb{C}}
\newcommand{\Z}{\mathbb{Z}}
\newcommand{\T}{\mathbb{T}}
\newcommand{\N}{\mathbb{N}}
\newcommand{\M}{\mathscr{M}_\zeta}
\newcommand{\Mk}{\mathscr{M}_{\zeta^k}}
\newcommand{\Mn}{\mathscr{M}_{\zeta^n}}
\newcommand{\Mm}{\mathscr{M}_{\zeta_m}}
\newcommand{\Msm}{\mathscr{M}_{\sigma_m}}

\newcommand{\vect}[1]{\boldsymbol{\mathbf{#1}}}
\newcommand{\tM}{\mathscr{M}}

\newcommand\restr[2]{{
  \left.\kern-\nulldelimiterspace
  #1
  \right|_{#2}}}

\begin{document}

\title{Uniformly distributed orbits in $\mathbb{T}^d$ and singular substitution dynamical systems}
\author{Rotem Yaari}
\address{Rotem Yaari, Department of Mathematics,
Bar-Ilan University, Ramat-Gan, Israel}
\email{rotemyaari@gmail.com}

\begin{abstract}
We find sufficient conditions for the singularity of a substitution $\mathbb{Z}$-action's spectrum, which generalize a result of Bufetov and Solomyak, and we also obtain a similar statement for a collection of substitution $\mathbb{R}$-actions, including the self-similar one.
To achieve this, we first study the distribution of related toral endomorphism orbits.
In particular, given a toral endomorphism and a vector $\mathbf{v}\in\mathbb{Q}^d$, we find necessary and sufficient conditions for the orbit of $\omega\mathbf{v}$ to be uniformly distributed modulo $1$ for almost every $\omega\in\mathbb{R}$.
We use our results to find new examples of singular substitution $\mathbb{Z}$- and $\mathbb{R}$-actions.
\end{abstract}



\maketitle

\section{Introduction}
While the discrete spectrum of substitution dynamical systems has been heavily studied, e.g., \cite{dekking1978spectrum, host1986valeurs, ferenczi1996substitution, host1992representation, barge2002coincidence, hollander2003two}, less is known on the existence (and absence) of the absolutely continuous component. Primitive substitution $\Z$- and $\R$-actions always possess a nontrivial singular component \cite[Theorem 2]{dekking1978mixing}, \cite[Theorem 2.2]{clark2003size}, but nevertheless an absolutely continuous component may exist; examples are provided by the Rudin-Shapiro substitution and its generalizations \cite{queffelec2010substitution, frank2003substitution, chan2016substitution}.
In general, it is hard to determine whether the spectrum is purely singular. In the case of a constant length substitution, Bartlett developed further the work of Queff{\'e}lec \cite{queffelec2010substitution} and obtained an algorithm for computing the spectrum of a substitution, which he used to find examples of substitutions with purely singular spectrum \cite{bartlett2018spectral}. Berlinkov and Solomyak provided a sufficient condition for the singularity of the spectrum, in terms of the eigenvalues of the substitution matrix \cite{berlinkov2019singular}. In the non-constant length self-similar $\R$-action case, Baake et al.\ \cite{baake2016pair, baake_2019, baake2018spectral, baake2019renormalisation}
developed new techniques which they used to obtain sufficient conditions for the singularity of the closely related diffraction spectrum, and they explored some examples (see Remark \ref{remark Baake}(c)).

In \cite{bufetov2014modulus,bufetov2020spectral}, Bufetov and Solomyak introduced the spectral cocycle associated with a substitution (or more generally, an S-adic shift), and used it to obtain sufficient conditions for a substitution $\R$-action to have purely singular spectrum, see \cite[Corollaries 4.5 and 4.7]{bufetov2020spectral}.
However, it is difficult to find examples of singular substitution $\R$-actions based directly on these results, since the conditions are given in terms of the \emph{pointwise} upper Lyapunov exponent, which is rather hard to compute.
The situation is better in the $\Z$-action case studied in \cite{bufetov2020singular} by the same authors.
Using the uniform distribution of the orbit of the diagonal vector $\omega \vec{1}$ (where $\vec{1} = (1,\ldots,1)^t$), for Lebesgue-almost every (a.e.) $\omega\in\R$, under some related toral endomorphism,
they were able to replace the condition on the pointwise exponent with a condition on a
system-dependent exponent which is easier to estimate, thereby making the singularity conditions easier to verify.
To establish this uniform distribution, the authors applied a theorem of Host \cite{host2000some}.
Whereas Host’s theorem is relatively general, the authors' interest lies only in the distribution of the mentioned orbit, and its use required them to add the assumption that the characteristic polynomial of the substitution matrix is irreducible over $\Q$, which, as we will see, is not necessary in that specific case.

Motivated by the connection between uniform distribution and singularity of substitutions, we study in Section \ref{ud section} (after recalling some definitions and results on uniformly distributed sequences and linear recurrences) the distribution of certain toral endomorphism orbits.
In particular, we give conditions that are both necessary and sufficient for the orbit of $\omega\vec{1}$ to be uniformly distributed in the torus for a.e.\ $\omega\in\R$.
This is done by exploiting the connection between toral endomorphism orbits and linear recurrences, which allows us to use a powerful theorem of van der Poorten \cite{van1981some} and Evertse \cite{evertse1984sums}.
It turns out that in the case of a reducible characteristic polynomial, the singularity of a given substitution $\Z$- or $\R$-action depends only on a component of the spectral cocycle, which is obtained by a restriction to what we call the \emph{minimal subspace} of a vector, see Section \ref{minimal subspace section}.
Then, following some of the ideas of \cite{bufetov2020singular},
we obtain in Section \ref{substitutions section} sufficient conditions for a substitution $\Z$-action, and for a collection of $\R$-actions including the self-similar one, to have purely singular spectrum,
without assuming irreducibility or Bohr-almost periodicity, and without using the theorems of Host and Sobol (see Remark \ref{remark Baake}(c)).
In addition, these conditions depend on a non-pointwise Lyapunov exponent,
enabling us to explore new examples of reducible non-Pisot substitutions with singular spectrum in Section \ref{examples section}.

\section{Uniformly distributed sequences and linear recurrences}\label{ud section}
Recall that a sequence $(\mathbf{x}_n)_{n=0}^\infty\subset\mathbb{R}^d$ is said to be \emph{uniformly distributed modulo $1$} (abbreviated u.d.\ mod $1$) if for every choice of intervals $I_1,\dots,I_d \subseteq [0,1)$ we have \[ \lim_{N\to\infty}\frac{|\{0\le n < N:\mathbf{x}_n\bmod 1\in I_1\times\cdots\times I_d\}|}{N} = |I_1|\cdots|I_d|\]
(where $\mathbf{x}\bmod 1$ stands for the vector of entrywise fractional part of $\mathbf{x}$), or equivalently, if for every $\Z^d$-periodic continuous function $f:\R^d\to\C$,
\[ \lim_{N\to\infty}\frac{1}{N}\sum_{n=0}^{N-1}f(\vect{x}_n) = \int_{\T^d} f \,dm_d,\]
where $m_d$ is the $d$-dimensional (normalized) Haar measure. Note that we use the same notation for $f$ and for the induced function on the $d$-dimensional torus $\T^d$.

The following is a straightforward consequence of the well-known Weyl's criterion.
\begin{proposition}[{\cite[Chapter 1, Theorem 6.3]{kuipers2012uniform}}]\label{h thm}
A sequence $(\mathbf{x}_n)_{n=0}^\infty\subset\mathbb{R}^d$ is u.d.\ mod $1$ if and only if for every non-zero $\mathbf{h}\in\mathbb{Z}^d$ the sequence of real numbers $(\langle\mathbf{x}_n, \mathbf{h}\rangle)_{n=0}^\infty$ is u.d.\ mod $1$.
\end{proposition}

The next result, which is a consequence of a theorem of Koksma, will also be useful.
\begin{theorem}[{\cite[Chapter 1, Corollary 4.3]{kuipers2012uniform}}]\label{metric thm}
Let $(x_n)_{n=0}^\infty$ be a real sequence. Suppose that there exist $N\in\N$ and $\delta > 0$ such that $\lvert x_n - x_m\rvert \ge \delta$ for every $n,m>N$, $n\neq m$, then the sequence $(x_n\omega)_{n=0}^\infty$ is u.d.\ mod $1$ for a.e.\ $\omega\in\R$.
\end{theorem}

The next definition is motivated by the theory of linear recurrence sequences, where a similar notion determines many properties
of the set of zeros of such sequences, see Definition \ref{degeneate relation def} and the discussion that follows.
\begin{definition}\label{deg matrix}
A non-singular matrix $A\in M_d(\mathbb{Z})$ is called \emph{degenerate} if it has two distinct eigenvalues whose ratio is a root of unity, and otherwise it is called \emph{non-degenerate}.
\end{definition}

We can now state our first main result.
\begin{theorem}\label{mainprop}
Let $A\in M_d(\mathbb{Z})$ be non-singular and let $\mathbf{v}\in\mathbb{Q}^d$. The sequence $(A^n\omega\mathbf{v})_{n=0}^\infty$ is u.d.\ mod $1$ for a.e.\ $\omega\in\mathbb{R}$ if and only if $A$ is non-degenerate with no eigenvalues that are roots of unity and the vectors $\mathbf{v},A\mathbf{v},\dots,A^{d-1}\mathbf{v}$ are linearly independent.
\end{theorem}

\begin{example}
Let $\begingroup
\renewcommand*{\arraystretch}{0.7}
\setlength\arraycolsep{3pt}
A = \begin{pmatrix}
    2 &   \\
      & 3 \\
  \end{pmatrix} \endgroup$, then by Theorem \ref{mainprop},
  $\begingroup
\renewcommand*{\arraystretch}{0.7}
\setlength\arraycolsep{3pt}
(A^n\begin{pmatrix}
    \omega \\ \omega
  \end{pmatrix})_{n=0}^\infty \endgroup$ is u.d.\ mod $1$ for a.e.\ $\omega\in\mathbb{R}$. Notice that it does not follow from the fact that $A$ is an ergodic transformation of $\T^2$ (with respect to Haar measure), since the diagonal $\{(\omega, \omega)\}$ has Haar measure zero.
\end{example}

Before proving this result, we recall a few definitions and results on linear recurrence relations.

\begin{definition}
A \emph{linear recurrence relation}, or simply a \emph{recurrence relation}, is an expression of the form
\begin{equation} \label{relation}
u_n = \sum_{i=0}^{d-1}\alpha_i u_{n-d+i}    
\end{equation}
for some $\alpha_0,\dots,\alpha_{d-1}\in\mathbb{C}$, $\alpha_0\neq0$, and we say that the recurrence relation is of \emph{order $d$}.  The \emph{companion polynomial} associated with the recurrence relation \eqref{relation} is the polynomial
$x^d - \sum_{i=0}^{d-1}\alpha_i x^{i}$,
and its roots are the \emph{roots of the recurrence}.
A \emph{linear recurrence sequence}, or simply a \emph{recurrence sequence}, is a complex sequence that satisfies some recurrence relation. A recurrence sequence is of \emph{order $d$} if the recurrence relation of minimal order that it satisfies is of order $d$.
\end{definition}
Notice that a recurrence sequence of order $d$ is determined by its first $d$ terms, which are called the \emph{initial values} of the sequence.

The following is one of the most fundamental facts about recurrence relations.
\begin{theorem}[{see \cite[Subsection 1.1.6]{everest2003recurrence}}] \label{solution}
Denote by $\lambda_1,\dots,\lambda_m$ the distinct roots of the recurrence relation \eqref{relation} and by $n_1,\dots,n_m$ their respective multiplicities.
The sequences that satisfy this recurrence relation are exactly the sequences $(x_n)_{n=0}^\infty\subset\mathbb{C}$ of the form
\begin{equation*}
    x_n = \sum_{i=1}^m p_i(n)\lambda_i^n, \quad n\in\mathbb{N}
\end{equation*}
\emph{(we define $\mathbb{N}$ to include $0$)}, where $p_i$ is a polynomial of degree $\deg p_i< n_i$ for $i=1,\dots,m$.
\end{theorem}

Many questions are concerned with the set of zeros of a given recurrence sequence. These questions and their generalizations have led to the following definitions.
\begin{definition}
The \emph{total multiplicity} of a complex sequence $(x_n)_{n=0}^\infty$ is defined to be
\[ |\{(n,m)\in\mathbb{N}\times\mathbb{N}:n\neq m,\, x_n=x_m\}|. \]
\end{definition}

\begin{definition}\label{degeneate relation def}
A recurrence relation is called \emph{degenerate} if it has two distinct roots whose ratio is a root of unity. Otherwise, the recurrence relation is called \emph{non-degenerate}.
\end{definition}

Given a degenerate recurrence relation, it is easy to construct a corresponding sequence with infinitely many zeros: take $\lambda^{n}-(\rho\lambda)^{n}$, where $\lambda, \rho \lambda$ are two distinct roots of the recurrence and $\rho$ is a root of unity.
The following deep theorem shows that under reasonable assumptions, this is the only case in which we can construct such a sequence from a recurrence relation.
\begin{theorem}[{van der Poorten} {\cite{van1981some}}, {Evertse} {\cite[Corollary 4]{evertse1984sums}}]\label{Evertse}
Suppose $(x_n)_{n=0}^\infty$ is a sequence of algebraic numbers that satisfies a non-degenerate recurrence relation. If the sequence is not of the form
$(x_n) = (c\rho^n)$,
for some constant $c$ and a root of unity $\rho$, then the total multiplicity of the sequence is finite.
\end{theorem}

The next definition provides the connection between the theory of linear recurrence relations and the rest of the topics discussed in the current paper.
\begin{definition}
Let $A\in M_d(\mathbb{Z})$ be non-singular and let $ x^d - \sum_{i=0}^{d-1}\alpha_i x^{i} $ be its characteristic polynomial. The \emph{recurrence relation associated with $A$} is $u_n = \sum_{i=0}^{d-1}\alpha_i u_{n-d+i}$. \end{definition}
Note that the companion polynomial associated with the recurrence relation is the characteristic polynomial of $A$, so $A$ is degenerate (see Definition \ref{deg matrix}) if and only if its associated recurrence relation is degenerate. Moreover, since $A$ is an integer matrix,
$\alpha_0,\dots,\alpha_{d-1}$ are integers as well.

We will need the following two lemmas.
\begin{lemma}\label{rec lemma}
Let $A\in M_d(\mathbb{Z})$ be non-singular and let $\mathbf{v}\in\mathbb{Q}^d$. Suppose $\mathbf{v},A\mathbf{v},\dots,A^{d-1}\mathbf{v}$ are linearly independent, then a sequence $(x_n)_{n=0}^\infty\subset\mathbb{Q}$ satisfies the recurrence relation associated with $A$ if and only if there exists $\mathbf{s}\in\mathbb{Q}^{d}$ such that
$x_n = \langle A^{n}\mathbf{v},\mathbf{s}\rangle$ for every $n\in\mathbb{N}$, and the sequence is identically zero if and only if $\mathbf{s}=\vect{0}$. 
\end{lemma}
\begin{proof}
If $(x_n) = (\langle A^{n}\mathbf{v},\mathbf{s}\rangle)$, the first part of the claim follows immediately from Cayley-Hamilton theorem.
Conversely, since $\mathbf{v},A\mathbf{v},\dots,A^{d-1}\mathbf{v}$ are linearly independent, there exists $\mathbf{s}\in\mathbb{Q}^d$ such that $\langle A^{i}\mathbf{v},\mathbf{s}\rangle =x_i$ for $i=0,\dots,d-1$.
The sequences $(\langle A^{n}\mathbf{v}, \mathbf{s}\rangle)_{n=0}^\infty$ and $(x_n)_{n=0}^\infty$ satisfy the same recurrence relation and have the same initial values, so they must be equal. The last part is clear from the linear independence of $\mathbf{v},A\mathbf{v},\dots,A^{d-1}\mathbf{v}$.
\end{proof}

\begin{lemma}\label{degenerate lemma}
Suppose that \eqref{relation} is a degenerate recurrence relation with $\alpha_0,\dots,\alpha_{d-1}\in\Z$. Then there exists a sequence of integers, which is not identically zero, satisfies \eqref{relation} and has an arithmetic subsequence of zeros.
\end{lemma}

\begin{proof}
Let $\lambda,\rho\lambda$
be two distinct roots of the recurrence, where $\rho$ is a root of unity of order $k$.
Using the recurrence relation \eqref{relation}, we can find integers $\beta_{i,j}$, $0\le i,j \le d-1$, such that every sequence $(x_n)_{n=0}^\infty$ that satisfies this recurrence relation also satisfies
\begin{equation}\label{substituted}
x_{i\cdot k}=\sum_{j=0}^{d-1}\beta_{i,j}x_{j}, \quad i=0,\dots,d-1.
\end{equation}
Define $B=(\beta_{i,j})_{0\le i,j\le d-1}\in M_{d}(\mathbb{Z})$ and
a sequence $(y_{n})_{n=0}^{\infty}$ by $y_{n}=\lambda^{n}-(\rho\lambda)^{n}$,
and notice that it is not identically zero, it satisfies the recurrence relation \eqref{relation} by Theorem \ref{solution} and it vanishes on the set $\{0,k,2k,\dots\}$. Thus, it follows from \eqref{substituted} that
\[
B\cdot(y_{0},\dots,y_{d-1})^{t}=(y_{0},y_{k},\dots,y_{(d-1)k})^{t}=0,
\]
and hence $\det B=0$ and there exists a non-zero vector $(z_{0},\dots,z_{d-1})^{t}\in\mathbb{Z}^{d}\cap\ker B$.
Let $(z_{n})$ be the recurrence sequence defined by these initial values and the recurrence relation \eqref{relation}. By \eqref{substituted},
\[
(z_{0},z_{k},\dots,z_{(d-1)k})^{t}=B\cdot(z_{0},\dots,z_{d-1})^{t}=0,
\]
and since $(z_{kn})_{n=0}^{\infty}$ is also a linear recurrence sequence of order at most $d$ (see \cite[Theorem 1.3]{everest2003recurrence}), this
subsequence must be identically zero.
\end{proof}

\begin{proof}[Proof of Theorem \ref{mainprop}]
Since we can replace $\mathbf{v}$ by an integer vector with the same span, we can assume without loss of generality that $\vect{v}\in\Z^d$. First we prove the sufficiency of the conditions.
If we prove that for every non-zero $\mathbf{h}\in\Z^d$, $(\langle A^{n}\omega\mathbf{v},\mathbf{h}\rangle)_{n=0}^\infty$ is u.d.\ mod $1$ for a.e.\ $\omega\in\R$, then the set of $\omega$'s that work for all $\mathbf{h}$'s is also of full measure, and by Proposition \ref{h thm} we are done.
Fix a non-zero $\mathbf{h}\in\Z^d$. By Lemma \ref{rec lemma}, $(\langle A^{n}\mathbf{v},\mathbf{h}\rangle)_{n=0}^\infty$ is not identically zero and satisfies the recurrence relation associated with $A$. Since no eigenvalue of $A$ is a root of unity, Theorems \ref{Evertse} and \ref{solution} imply that the total multiplicity of this integer sequence is finite, and the sufficiency of the conditions follows from Theorem \ref{metric thm} (we can take $\delta=1$). 

Conversely, suppose first that $\mathbf{v},A\mathbf{v},\dots,A^{d-1}\mathbf{v}$
are linearly dependent. Notice that
\[\{A^n\omega\mathbf{v} :\omega\in\R,\,n\in\N\} \subseteq \Span\{\mathbf{v},A\mathbf{v},\dots,A^{d-1}\mathbf{v}\},\]
and since this subspace is spanned by at most $d-1$ integer vectors, the set of its fractional parts is not dense in $\T^{d}$, let alone contains a u.d.\ sequence mod $1$. We can assume for the rest of the proof that $\mathbf{v},A\mathbf{v},\dots,A^{d-1}\mathbf{v}$ are linearly independent.

Suppose now that $A$ is degenerate. By Lemma \ref{degenerate lemma}, we can take a sequence of integers $(x_n)_{n=0}^\infty$, which is not identically zero, satisfies the recurrence relation associated with $A$ and such that $x_{kn} = 0$ for some $k\ge 2$ and every $n$.
By Lemma \ref{rec lemma} there exist $\mathbf{h}\in\mathbb{Z}^{d}\setminus\{\mathbf{0}\}$ and $c\in\mathbb{N}\setminus\{0\}$ such that $\langle A^{n}\mathbf{v},\mathbf{h}\rangle=cx_{n}$ for every $n$. Consequently, for every $\omega\in\mathbb{R}$ we have
\[
\limsup_{N\to\infty}\frac{|\{0\le n<N:\langle A^{n}\mathbf{v},\mathbf{h}\rangle\omega\bmod 1\in[0,\frac{1}{2k}]\}|}{N}\ge\frac{1}{k},
\]
so $\langle A^{n}\omega\mathbf{v},\mathbf{h}\rangle$ is not u.d.\
mod $1$ and again by Proposition \ref{h thm}, $A^{n}\omega\mathbf{v}$ is also not u.d.\ mod $1$.

Finally, if $\rho$ is an eigenvalue of $A$ which is also a root
of unity, then so is $\overline{\rho}$. If $\rho\neq\overline{\rho}$,
then $A$ is degenerate and since we already considered this case
we may assume that $\rho=\pm1$. Proceeding as before, $(\langle A^{n}\mathbf{v},\mathbf{h}\rangle)=(c\rho^{n})$ for some $\mathbf{h}\in\mathbb{Z}^{d}\setminus\{\mathbf{0}\}$ and $c\in\mathbb{N}\setminus\{0\}$.
It follows that for every $\omega\in\mathbb{R}$, \[(\langle A^{n}\omega\mathbf{v},\mathbf{h}\rangle)_{n=0}^{\infty}\subseteq\{\pm c\omega\},\]
and once again by Proposition \ref{h thm}, $(A^{n}\omega\mathbf{v})_{n=0}^\infty$
is not u.d.\ mod $1$.
\end{proof}

\begin{corollary}
If $(A^{n}\omega\mathbf{v})$
is u.d.\ mod $1$ for some $\omega\in\mathbb{R}$, then the same is true for a.e.\ $\omega\in\mathbb{R}$.
\end{corollary}
\begin{proof}
We saw in the previous proof that if one of the conditions of Theorem \ref{mainprop} does not hold then for every $\omega\in\mathbb{R}$, $(A^{n}\omega\mathbf{v})_{n=0}^\infty$ is not u.d.\ mod $1$.
\end{proof}
\begin{corollary}\label{k corollary}
If $(A^{n}\omega\mathbf{v})$
is u.d.\ mod $1$ for a.e.\ $\omega\in\mathbb{R}$, then so is $(A^{kn+\ell}\omega\mathbf{v})_{n=0}^\infty$ for every $k \ge 1$ and $\ell\in\N$.
\end{corollary}
\begin{proof}
By Theorem \ref{mainprop} we just need to show that $A^\ell\mathbf{v},A^{k+\ell}\mathbf{v},\dots,A^{k(d-1)+\ell}\mathbf{v}$ are linearly independent. Suppose $\langle A^{ki + \ell}\vect{v}, \vect{h}\rangle = 0$ for $i=0,\dots,d-1$ and some $\vect{h}\in\Z^d$, then $(\langle A^n\vect{v},\vect{h}\rangle)$ has an arithmetic subsequence of zeros, but since the associated recurrence relation is non-degenerate, $(\langle A^n\vect{v},\vect{h}\rangle)$ must be identically zero (see \cite[Corollary C.1]{exponential_diophantine}), and thus $\vect{h} = \vect{0}$.
\end{proof}
\begin{remark}
(a) Meiri proved that if an integer sequence $(x_n)$ satisfies a non-degenerate recurrence relation that has no roots that are roots of unity, then in fact $(\omega x_n)$ is u.d.\ mod $1$ for $\mu$-a.e.\ $\omega$, where $\mu$ belongs to some collection of Borel measures on $\T$, including Lebesgue measure \cite[Theorem 5.2]{meiri1998entropy}. For Lebesgue measure, we gave a simple (one-line) proof of this fact, relying on the powerful result of van der Poorten and Evertse, whereas the proof of Meiri's result is considerably more complicated, and uses $p$-adic analysis instead.

(b) Pushkin obtained a somewhat similar result, showing that
given a connected analytic manifold in $\R^d$ that is not contained in any hyperplane, Lebesgue-a.e.\ vector in that manifold is absolutely normal \cite[Theorem 2]{pushkin1992borel}.
\end{remark}

\begin{proposition} \label{pf prop}
Suppose that $A\in M_d(\Z)$ has a unique dominant eigenvalue $\theta_1 > 1$, and that its characteristic polynomial is irreducible over $\Q$. Let $\mathbf{v} = \sum_{i=1}^d c_i \vect{v}_i\in\R^d$ where $\vect{v}_1,\dots,\vect{v}_d$ are the eigenvectors of $A$, $\vect{v}_1$ corresponds to $\theta_1$,\; $c_1,\dots,c_d \in\C$ and $c_1 \neq 0$.
Then $(A^{kn+\ell}\omega\mathbf{v})_{n=0}^\infty$ is u.d.\ mod $1$ for every $k\ge 1$, $\ell\in\N$ and a.e.\ $\omega\in\R$.
\end{proposition}
\begin{proof}
First let us show that the entries of $\vect{v}_1$ are rationally independent. Suppose that $\langle \vect{v}_1,\vect{h}\rangle = 0$ for some $\vect{h}\in\Z^d$, then also
\[ 0 = \langle A^n \vect{v}_1, \vect{h} \rangle = \langle \vect{v}_1, (A^t)^n \vect{h} \rangle, \]
and hence $\vect{h},\dots, (A^t)^{d-1} \vect{h}$ must be linearly dependent. Therefore, they span an $A^t$-invariant $\Q^d$-subspace of dimension at most $d-1$, and the characteristic polynomial of $A^t$ restricted to this subspace divides the characteristic polynomial of $A$, which means $\vect{h}=\vect{0}$.

Next, let $k\ge 1$, $\ell\in\N$ and $\vect{h}\in\Z^d\setminus\{\vect{0}\}$, and consider the real sequence $(\langle A^{kn+\ell} \vect{v}, \vect{h}\rangle)_{n=0}^\infty$. Since $\langle c_1\vect{v}_1, \vect{h}\rangle\neq 0$ and $\theta_1$ is the unique dominant eigenvalue of $A$, we have
\[ \frac{\langle A^{k(n+1)+\ell} \vect{v}, \vect{h}\rangle}{\langle A^{kn+\ell} \vect{v}, \vect{h}\rangle} \underset{n}{\to} \theta_1^k,\]
so in particular $\lvert \langle A^{km+\ell} \vect{v}, \vect{h}\rangle - \langle A^{kn+\ell} \vect{v}, \vect{h}\rangle\rvert > 1$ for every sufficiently large $n$ and every $m >n$. By Theorem \ref{metric thm}, $(\langle A^{kn+\ell} \omega \vect{v}, \vect{h}\rangle)$ is u.d.\ mod $1$ for a.e.\ $\omega\in\R$,
and we conclude by repeating the argument from the beginning of the proof of Theorem \ref{mainprop}.
\end{proof}

\section{The minimal subspace}\label{minimal subspace section}
While the distribution of some sequences of the form $(A^n\omega \vect{v})$ is just far from uniform, 
another reason why such a sequence may fail to be uniformly distributed in the torus $\T^d$, is that it is actually uniformly distributed in a proper subtorus. We would not like to exclude such orbits, since uniform distribution in a subtorus will suffice for our purpose. Thus, we introduce the following definition, which aims at finding the right subspace to look at for this matter. 
\begin{definition}
Let $A\in M_d(\Z)$ and $\vect{v}\in\mathbb{R}^d\setminus\{\vect{0}\}$. The \emph{minimal subspace of $\vect{v}$} (with respect to $A$) is $\rspan{W}<\R^d$, where $W < \Q^d$ is the minimal $A$-invariant subspace (over $\mathbb{Q}$), such that $\vect{v}\in\rspan W$.
\end{definition}
The minimal subspace of a vector is determined by an interplay between the dimension of $\rspan\{\vect{v},A\vect{v},\dots\}$ and the extent to which the vector $\vect{v}$ is irrational. For example, the minimal subspace of an eigenvector $\vect{v}$ with rational entries is just $\R\vect{v}$ (which accounts for the fact that the image of $\R \vect{v}$ in $\T^d$ is a one-dimensional torus), while at the other extreme the minimal subspace of an eigenvector with rationally independent entries is $\R^d$ (which accounts for the fact that the image of $\R\vect{v}$ in $\T^d$ is not contained in any subtorus).

The following lemma asserts that, as implied in the definition, there is a unique minimal subspace $W < \Q^d$ with these properties, and hence the minimal subspace is unique as well.
\begin{lemma}
Let $W_1,W_2 < \Q^d$, then $\rspan(W_1\cap W_2) = \rspan W_1\cap\rspan W_2$. 
\end{lemma}
\begin{proof}
Clearly, $\dim_\R \rspan W \le \dim_\Q W$ for any $W< \Q^d$, and since we can define a non-singular matrix (over both fields) with columns that contain a basis of $W$, the dimensions are equal.
The inclusions $\rspan(W_1\cap W_2) \subseteq \rspan W_1\cap\rspan W_2$ and $\rspan(W_1 + W_2) \subseteq \rspan W_1 + \rspan W_2$ are clear, and the lemma follows from the identity $\dim (U \cap V) = \dim U + \dim V - \dim (U + V)$.
\end{proof}

\begin{example} \label{minimal subspace example}
{(a)} Let $A\in M_d(\Z)$. If $\vect{v}\in\Q^d\setminus\{\vect{0}\}$, then its minimal subspace is the cyclic subspace $\Span_\R \{\vect{v},A\vect{v},\dots,A^{d-1}\vect{v}\}$. It is invariant by Cayley-Hamilton theorem, and $\vect{v},A\vect{v},\dots,A^{r-1}\vect{v}$ is a basis for this subspace, where $r\le d$ is the maximal integer such that these vectors are linearly independent.

{(b)} Let $ A = \begin{pmatrix}
    1 & 1 & 1 \\
    1 & 1 & 1 \\
    0 & 4 & 0 \\
  \end{pmatrix}$ and $\vect{v} = (\sqrt{5}+1, \sqrt{5}+1, 4)^t$. Since $\vect{v}$ is clearly not a (real) multiple of a vector in $\Q^3$, its minimal subspace must be at least two-dimensional. It follows that the $A$-invariant subspace $\rspan\{(1,1,0)^t, \, (0,0,1)^t\}$ is the minimal subspace of $\vect{v}$.

{(c)} Generalizing the last example, suppose that $A\in M_d(\Z)$ is primitive. Let $\vect{u}\in\R^d$ be its Perron-Frobenius eigenvector, corresponding to the Perron-Frobenius eigenvalue $\theta_1$, and let $p_{\theta_1}$ be the minimal polynomial of $\theta_1$ over $\Q$.
Any $A$-invariant $\Q$-subspace $W$ with $\vect{u}\in\rspan{W}$ must have $\dim W \ge \deg(p_{\theta_1})$.
Since $\theta_1$ is a simple eigenvalue, it follows from the primary decomposition theorem (see \cite[Chapter 6, Theorem 12]{hoffmann1971linear}) that the minimal and characteristic polynomials of $A$ restricted to $U\coloneqq\ker p_{\theta_1}(A)$ equal $p_{\theta_1}$.
Thus, $U$ is the minimal subspace of $\vect{u}$, and in fact, of any non-zero $\vect{v}\in U$.
\end{example}

\begin{lemma}\label{basis lemma}
Let $A\in M_d(\Z)$ and let $V$ be the minimal subspace of some non-zero $\vect{v}\in\R^d$.
There exists a basis of integer vectors for $V$, such that every integer vector in $V$ has integer coordinates with respect to that basis. In particular, the map $\restr{A}{V}$, written in that basis, is an integer matrix (rather than rational).
\end{lemma}
\begin{proof}
Notice that $\Z^d\cap V$ is a subgroup of $\Z^d$, and thus it is free abelian. It is easy to check that a basis of this free abelian group is also a basis of $V$ which meets all the above requirements.
\end{proof}
Such a basis will be called a \emph{lattice basis of $V$}. 

\begin{definition} \label{ud in minimal subspace}
Let $\mathcal{B}$ be a lattice basis of $V$, and consider the isomorphism $\varphi_\mathcal{B}:V\to \R^{r}$ (where $r=\dim V$) that maps a vector to its coordinate vector $\vect{v}\mapsto [\vect{v}]_\mathcal{B}$. A sequence $(\vect{x}_n)\subset V$ is said to be \emph{u.d.\ mod $1$ in $V$} if the sequence $(\varphi_\mathcal{B}(\vect{x}_n))$ is u.d.\ mod $1$ in $\R^{r}$.
\end{definition}

\begin{remark}
It is not hard to see that this definition is independent of the choice of the lattice basis, and that $(\vect{x}_n)$ is u.d.\ mod $1$ in $V$ if and only if for every lattice basis $\mathcal{B}$ and every $\Z^d$-periodic continuous function $f:V\to\C$,
$\lim_{N\to\infty}\frac{1}{N}\sum_{n=0}^{N-1}f(\vect{x}_n) = \int_{\T^{r}}f\circ \varphi^{-1}_\mathcal{B} \,dm_{r}$,
where $m_r$ is the $r$-dimensional Haar measure.
\end{remark}
We call a vector $\vect{v}$ positive and write $\vect{v} > \vect{0}$ if it is entrywise positive, and the same applies to matrices.
\begin{corollary} \phantomsection \label{ud corollary}
\begin{enumerate}[font=\normalfont, label=(\alph*)]
    \item \label{ud corollary a} Let $A\in M_d(\Z)$ and $\mathbf{v}\in\Q^d\setminus\{\vect{0}\}$.
    Let $V = \Span_\R \{\vect{v},A\vect{v},\dots,A^{d-1}\vect{v}\}$, and suppose that $\restr{A}{V}$ is non-singular. The sequence $(A^{kn+\ell}\omega\mathbf{v})_{n=0}^\infty$ is u.d.\ mod $1$ in $V$ for every $k\ge 1$, $\ell\in\N$ and a.e.\ $\omega\in\mathbb{R}$ if and only if $\restr{A}{V}$ is non-degenerate with no eigenvalues that are roots of unity.
    \item \label{ud corollary b} Suppose that $A\in M_d(\Z)$ is primitive with a Perron-Frobenius eigenvalue $\theta_1$. Denote by $p_{\theta_1}$ the minimal polynomial of $\theta_1$ over $\Q$ and let $\mathbf{v}\in \ker p_{\theta_1}(A)$, $\mathbf{v} > \vect{0}$. Then $(A^{kn+\ell}\omega\mathbf{v})_{n=0}^\infty$ is u.d.\ mod $1$ in $\ker p_{\theta_1}(A)$ for every $k\ge 1$, $\ell\in\N$ and a.e.\ $\omega\in\mathbb{R}$.
\end{enumerate}
\end{corollary}
\begin{proof}
(a) This is just the combination of Theorem \ref{mainprop}, Corollary \ref{k corollary} and Example \ref{minimal subspace example}(a).

(b) It is well-known that the \emph{Perron projection} $P$, defined by $P\vect{u} = \vect{u}$ for the Perron-Frobenius eigenvector $\vect{u}$ and $P\vect{w} = \vect{0}$ for any other generalized eigenvector, is a positive matrix (see for example, \cite[Chapter 8]{meyer2000matrix}). It follows that $P\vect{v} > \vect{0}$, so the $\vect{u}$-component of $\vect{v}$ is not $0$, and we conclude by combining Proposition \ref{pf prop} and Example \ref{minimal subspace example}(c).
\end{proof}

\section{Applications to substitutions} \label{substitutions section}
Let $\mathcal{A}=\{0,\dots,d-1\}$ be a finite alphabet with $d\ge 2$. 
A \emph{substitution} is a map $\zeta:\mathcal{A}\to\mathcal{A}^+$, where $\mathcal{A}^+=\bigcup_{n=1}^\infty \mathcal{A}^n$ is the collection of all finite words.
This map is extended to $\mathcal{A}^+$ and to $\mathcal{A}^{\Z}$ by concatenation, and these extensions are called substitutions and denoted by $\zeta$ as well.
The \emph{substitution dynamical system}, also sometimes called the \emph{substitution $\Z$-action}, is the space
\[ X_\zeta = \{ x\in\mathcal{A}^\mathbb{Z}:\text{every finite subword of } x \text{ is also a subword of } \zeta^n(a) \text{ for some }a\in\mathcal{A}\text{ and }n\in\N \}, \]
together with the left shift map on $\mathcal{A}^{\Z}$, restricted to $X_\zeta$.
To every substitution we associate its \emph{substitution matrix}, which is the $d\times d$ integer matrix $S_\zeta\in M_d(\mathbb{Z})$ whose $(i,j)$-th entry equals the number of times the letter $i$ appears in $\zeta(j)$, for every $0\le i,j\le d-1$. Note that $S_{\zeta^n}=S_\zeta^n$.
The substitution is \emph{primitive} if its substitution matrix is primitive,
and in that case, the substitution dynamical system is uniquely ergodic. We say that the substitution is \emph{periodic} if $X_\zeta$ contains a shift-periodic point,
and otherwise it is \emph{aperiodic}. For more details on substitutions see \cite{queffelec2010substitution, fogg2002substitutions}.
Given a positive vector $\vect{v}=(v_0,\dots,v_{d-1})^t\in\R^d$, the associated \emph{substitution $\R$-action} is the suspension flow over the substitution dynamical system, with the piecewise-constant roof function $f_{\vect{v}}:X_\zeta\to\R^+$  defined by $f_{\vect{v}}(x) = v_{x_0}$.
An equivalent and fruitful way to view this system is as a one-dimensional tiling space: there are $d$ tile types labeled from $0$ to $d-1$, the lengths of which are given by the entries of $\vect{v}$, and the tilings of $\R$ are determined by elements of $X_\zeta$, see \cite{berend1993there, solomyak1997dynamics, clark2003size}.
Two cases of particular interest arise when $\vect{v}$ is chosen to be the Perron-Frobenius eigenvector of $S_\zeta^t$, where the associated $\R$-action is then called \emph{self-similar}, and when $\vect{v}=\vec{1}$ (where $\vec{1} = (1,\ldots,1)^t$), which is closely related to the substitution $\Z$-action, see \cite[Lemma 5.6]{berlinkov2019singular}.

In \cite{bufetov2020spectral}, Bufetov and Solomyak define the \emph{spectral cocycle} that corresponds to $\zeta$, and it is further developed in \cite{bufetov2020singular}. In what follows, we will generalize their construction, while largely following their path.
For every $b\in\mathcal{A}$ denote 
$ \zeta(b) = u_1^b\ldots u_{\lvert\zeta(b)\rvert}^b$ (where $\lvert w\rvert$ stands for the length of the word $w$).  
First, define a matrix-valued function $\mathscr{M}_\zeta:\mathbb{R}^d\to M_d(\mathbb{C})$: let $\vect{\xi} = (\xi_0,\dots,\xi_{d-1})^t\in\mathbb{R}^d$, then $\mathscr{M}_\zeta(\vect{\xi})$ is the complex matrix whose $(b,c)$-th entry is
\begin{equation}\label{matrix function eq}
    \sum_{1\le j\le\lvert\zeta(b)\rvert,\, u_j^b = c} \exp(-2\pi i \sum_{k=1}^{j-1} \xi_{u_k^b}).
\end{equation}
When $\vect{\xi}$ is positive (which can always be assumed, since we only care about its value modulo $\Z^d$), we can use the viewpoint described earlier to think on the substitution $\R$-action associated with $\vect{\xi}$ as a one-dimensional tiling space. In these settings, if we take a tiling starting (from $0$) with a tile of type $b$, then in the substituted tiling, $\sum_{k=1}^{j-1} \xi_{u_k^b}$ is the location of the left endpoint of the $j$-th tile, which is of type $c$.
Thus, the summands in \eqref{matrix function eq} account for all left endpoint locations of tiles of type $c$ in $\zeta(b)$.

\begin{example}\label{matrix function example}
Let $\zeta$ be the substitution defined by $\zeta(0)=012,\enspace \zeta(1)=210,\enspace \zeta(2)=1111$ and denote $e(x)=\exp(2\pi i x)$. Then for every $\vect{\xi}\in\mathbb{R}^3$,
\[ \mathscr{M}_\zeta(\vect{\xi}) = \begin{pmatrix}
1 & e(-\xi_0) & e(-(\xi_0+\xi_1)) \\
e(-(\xi_1 + \xi_2)) & e(-\xi_2) & 1 \\
0 & 1 + e(-\xi_1) + e(-2\xi_1) + e(-3\xi_1) & 0
\end{pmatrix}. 
\] 
\end{example}

Note that $\mathscr{M}_\zeta (\vect{0})$ is just $S_\zeta^t$, and that $\M$ is $\mathbb{Z}^{d}$-periodic, so it descends to a function on $\mathbb{T}^{d}$. The function $\M$ gives rise to the \emph{spectral cocycle}, which is the matrix-valued function defined by
\begin{equation} \label{spectral cocycle}
   \M(\vect{\xi}, n)\coloneqq \M(E_{S_\zeta^t}^{n-1}\vect{\xi})\cdots\M(\vect{\xi}), \quad \vect{\xi}\in\T^d,\; n=1,2,\dots
\end{equation}
where $E_{S_\zeta^t}$ is the endomorphism of $\T^d$ induced by $S_\zeta^t$,
\[ E_{S_\zeta^t} (\vect{\xi}\bmod{\mathbb{Z}^d}) = S_\zeta^t \vect{\xi} \bmod{\mathbb{Z}^d},\quad \vect{\xi}\in\R^d\]
(notice that if $\det S_\zeta = 0$, $E_{S_\zeta^t}$ does not preserve Haar measure).
A computation shows that for every $n\ge 1$,
$\M(\vect{\xi},n) = \Mn(\vect{\xi})$.

We now turn to define the \emph{essential spectral cocycle} of a vector, which is the restriction of the spectral cocycle to a subtorus corresponding to the minimal subspace of the vector, written in appropriate coordinates.
In this way we will be able to take advantage of the uniform distribution of the vector's orbit that was established in the previous section.
Let $\vect{v}\in\R^d\setminus\{\vect{0}\}$. Let $V$ be its minimal subspace with respect to $S_\zeta^t$ and denote $\dim V = r$.
Fix a lattice basis $\mathcal{B}$ of $V$, and denote by $B$ the integer matrix that corresponds to ${S_\zeta^t|}_V$ in that basis. Assume that $B$ is non-singular and that no eigenvalue of $B$ is a root of unity, so unlike $E_{S_\zeta^t}$, the endomorphism $E_B$, induced by $B$ on $\T^{r}$, 
is measure-preserving and ergodic with respect to the (normalized) Haar measure $m_r$, see \cite[Corollary 2.20]{einsiedler2013ergodic}.
As before, let $\varphi_\mathcal{B}: V\to \R^r$ be the coordinate isomorphism, $\vect{\xi} \mapsto [\vect{\xi}]_\mathcal{B}$, and define $\tM_{\zeta,\vect{v}}: \R^{r}\to M_d(\C)$ by $\tM_{\zeta, \vect{v}} = \M\circ\varphi^{-1}_\mathcal{B}$.
Since $\mathcal{B}$ is composed of integer vectors, $\tM_{\zeta, \vect{v}}$ is $\mathbb{Z}^{r}$-periodic, so it descends to a function on $\mathbb{T}^{r}$.
The \emph{essential spectral cocycle of $\vect{v}$} is defined, similarly to \eqref{spectral cocycle}, to be the following matrix-valued function:
\[ \tM_{\zeta, \vect{v}}(\vect{x},n)\coloneqq \tM_{\zeta, \vect{v}}(E_B^{n-1}\vect{x})\cdots\tM_{\zeta, \vect{v}}(\vect{x}), \quad \vect{x}\in\T^r,\; n=1,2,\dots\]
Note that $\vect{x}\mapsto \M(\varphi^{-1}_\mathcal{B} (\vect{x}), n)$ is also $\Z^r$-periodic and $\tM_{\zeta, \vect{v}}(\vect{x},n) = \M(\varphi^{-1}_\mathcal{B} (\vect{x}), n)$. In addition, observe that in fact, $\tM_{\zeta, \vect{v}}$ only depends on the minimal subspace $V$, rather than on $\vect{v}$ itself; in particular, the essential spectral cocycles of vectors with the same minimal subspace are identical.

In what follows, $\lVert \cdot \rVert$ stands for a matrix norm on $M_d(\C)$. All the following claims are independent of the choice of the norm, since all such norms are equivalent. Therefore, for the rest of the paper we will use the Frobenius norm, which is submultiplicative.
The next lemma is a simple modification of \cite[Lemma 2.3]{bufetov2020singular}.
\begin{lemma} \label{integrable lemma}
For every $n\ge 1$, the function $\vect{x}\mapsto \log\lVert\tM_{\zeta, \vect{v}}(\vect{x}, n)\rVert$ is integrable over $(\mathbb{T}^{r},m_{r})$.
\end{lemma}
\begin{proof}
Notice that $\lVert\tM_{\zeta, \vect{v}}(\vect{x}, n)\rVert \le \lVert S_\zeta^n \rVert$.
Writing
\[ \lVert \tM_{\zeta, \vect{v}}(\vect{x}, n) \rVert^2 = \sum_{b,c} (\Mn(\varphi_\mathcal{B}^{-1}(\vect{x})))_{bc} \overline{(\Mn(\varphi_\mathcal{B}^{-1}(\vect{x})))_{bc}}\]
and observing that $\lVert \tM_{\zeta, \vect{v}}(\vect{0}, n) \rVert ^2 = \lVert S_{\zeta}^n \rVert ^2$, we see that $\lVert \tM_{\zeta, \vect{v}}(\vect{x}, n) \rVert ^2 $ is a nontrivial multivariate trigonometric polynomial with integer coefficients.
The integral $\int_{\T^r}\log\lVert\tM_{\zeta, \vect{v}}(\vect{x}, n)\rVert^2\, d m_r(\vect{x})$
is just the logarithmic Mahler measure of this polynomial, which is known to be at least $0$, see e.g.\ \cite{boyd_mahler}.
\end{proof}

By Furstenberg-Kesten theorem \cite{furstenberg1960products} (see also \cite{viana2014lectures}), the Lyapunov exponent exists, namely, the following limit exists and is constant for $m_r$-a.e.\ $\vect{x}\in\T^r$:
\[ \chi_{\zeta,\vect{v}} \coloneqq \lim_{n\to\infty}\frac{1}{n} \log\lVert\tM_{\zeta, \vect{v}}(\vect{x},n)\rVert,\]
and we call it the \emph{essential Lyapunov exponent of $\vect{v}$}.
It is independent of the choice of the norm and the basis $\mathcal{B}$. In addition, by Kingman's theorem (see for example, \cite[Theorem 3.3]{viana2014lectures}), the following identity holds:
\begin{equation} \label{inf identity}
    \chi_{\zeta,\vect{v}} = \inf_{k\ge 1} \frac{1}{k}\int_{\T^r} \log\lVert \tM_{\zeta, \vect{v}}(\vect{x}, k)\rVert \,dm_r(\vect{x}).
\end{equation}

\begin{remark}
In the case that the minimal subspace of $\vect{v}$ is $\R^d$, the spectral cocycle and the essential spectral cocycle of $\vect{v}$ coincide. In \cite{bufetov2020singular}, it is assumed that the characteristic polynomial of $S_\zeta$ is irreducible over $\Q$, so the minimal subspace of any non-zero vector is $\R^d$.
\end{remark}

\begin{proposition} \label{exponent inequality}
Let $\zeta$ be a substitution on $\mathcal{A} = \{0,\dots,d-1\}$ with $d\ge 2$. Let $\vect{v}\in\mathbb{R}^d\setminus \{\vect{0}\}$ and let $V$ be its minimal subspace. Denote by $S_\zeta$ the substitution matrix,
and suppose that ${S_\zeta^t|}_V$ is non-singular and has no eigenvalue that is a root of unity.
If $({(S_\zeta^t)}^{kn} \vect{w})_{n=0}^\infty$ is u.d.\ mod $1$ in $V$ for some $\vect{w}\in V$ and every $k\ge 1$, then
\begin{equation}\label{exponent bound eq}
\chi_\zeta^+ (\vect{w}) \coloneqq \limsup_{n\to\infty}\frac{1}{n} \log \lVert \M(\vect{w}, n)\rVert \le \chi_{\zeta,\vect{v}}.
\end{equation}
\end{proposition}
\begin{proof}
We closely follow the proof of Theorem 2.4 in \cite{bufetov2020singular}. For every $k\ge 1$,
\begin{align*}
    \chi_\zeta^+ (\vect{w}) & = \limsup_{n\to\infty}\frac{1}{nk} \log \lVert \Mk(\vect{w}, n)\rVert
    \le \limsup_{n\to\infty}\frac{1}{nk} \sum_{j=0}^{n-1} \log \lVert \Mk(E_{S_\zeta^t}^{kj}(\vect{w})) \rVert
    \\ & \le \lim_{\varepsilon\to 0} \limsup_{n\to\infty}\frac{1}{nk} \sum_{j=0}^{n-1} \log (\varepsilon + \lVert \M(E_{S_\zeta^t}^{kj}(\vect{w}), k) \rVert)
    = \lim_{\varepsilon\to 0} \frac{1}{k} \int_{\mathbb{T}^{r}} \log (\varepsilon + \lVert \tM_{\zeta, \vect{v}}(\vect{x},k)\rVert) \,dm_r(\vect{x}),
\end{align*}
where $r = \dim V$ and in the last equality we used the uniform distribution mod $1$ of $({(S_\zeta^t)}^{kn} \vect{w})$ in $V$.
Split the last integral into two parts, over $\{\lVert \tM_{\zeta, \vect{v}}(\vect{x},k)\rVert \ge \frac{1}{2}\}$ and $\{\lVert \tM_{\zeta, \vect{v}}(\vect{x},k)\rVert\in[0,\frac{1}{2})\}$. In the first domain the functions are uniformly bounded, and in the second we have
\[ \lvert \log (\varepsilon + \lVert \tM_{\zeta, \vect{v}}(\vect{x}, k)\rVert) \rvert \le \lvert \log \lVert \tM_{\zeta, \vect{v}}(\vect{x}, k)\rVert \rvert, \]
so by Lemma \ref{integrable lemma}, we can apply the dominated convergence theorem to obtain
\[ \chi_\zeta^+ (\vect{w}) \le \frac{1}{k} \int_{\mathbb{T}^r} \log\lVert \tM_{\zeta, \vect{v}}(\vect{x}, k)\rVert \,dm_r(\vect{x}).\]
The proof is now completed thanks to \eqref{inf identity}.
\end{proof}

Now we state our second main result.
\begin{theorem} \label{main theorem}
Let $\zeta$ be a primitive aperiodic substitution on $\mathcal{A} = \{0,\dots,d-1\}$ with $d\ge 2$. Denote by $S_\zeta$ the substitution matrix and by $\theta_1$ the Perron-Frobenius eigenvalue.
\begin{enumerate}[font=\normalfont, label=(\alph*)]
    \item Let $\vect{v}\in \ker p_{\theta_1}(S_\zeta^t)$, $\vect{v}>\vect{0}$, where $p_{\theta_1}$ is the minimal polynomial of $\theta_1$ over $\Q$.
    If \[ \chi_{\zeta,\vect{v}} < \frac{\log\theta_1}{2},\]
    then the substitution $\R$-action associated with $\vect{v}$ has purely singular spectrum. In particular, this is true for the self-similar $\R$-action associated with the Perron-Frobenius eigenvector of $S_\zeta^t$, and if the characteristic polynomial of $S_\zeta$ is irreducible, we can take any positive vector $\vect{v}\in\R^d$.
    \item Let $V = \rspan\{\vec{1},\dots,{(S_\zeta^t)}^{d-1}\vec{1}\}$, and suppose that ${S_\zeta^t|}_V$ is non-singular, non-degenerate and has no eigenvalue that is a root of unity.
    If
    \[
        \chi_{\zeta, \vec{1}} < \frac{\log\theta_1}{2},
    \]
    then the substitution $\mathbb{Z}$-action has purely singular spectrum.
\end{enumerate}
\end{theorem}
\begin{proof}
(a) It follows from Corollary \ref{ud corollary}\ref{ud corollary b} that $({(S_\zeta^t)}^{kn}(S_\zeta^t)^{\ell} \omega \vect{v})_{n=0}^\infty$ is u.d.\ mod $1$ in $U\coloneqq\ker p_{\theta_1}(S_\zeta^t)$ for every $k\ge 1$ and $\ell\in\N$ for a.e.\ $\omega\in\R$. It was observed in Example \ref{minimal subspace example}(c) that the minimal polynomial of ${S_\zeta^t|}_U$ is $p_{\theta_1}$, so ${S_\zeta^t|}_U$ is non-singular and has no eigenvalue that is a root of unity (otherwise $p_{\theta_1}$ would have been cyclotomic, but $\theta_1 > 1$). Thus, by Proposition \ref{exponent inequality},
\begin{equation}\label{exp p ineq 2}
    \chi_\zeta^+(E_{S_\zeta^t}^\ell\omega\vect{v}) < \frac{\log\theta_1}{2}
\end{equation}
for a.e.\ $\omega$.
Now the claim follows from \cite[Corollary 4.5(iii)]{bufetov2020spectral}
(notice that the additional assumption made there is needed only to prove \eqref{exp p ineq 2}, and see also Section $4.2$ in that paper), but let us briefly sketch the proof of this corollary for the sake of clarity.
Denote by $(\mathfrak{X}^{\vect{v}}_{\zeta},h_t,\widetilde{\mu})$ the suspension flow associated with $\vect{v}$.
A collection of $L^2(\mathfrak{X}^{\vect{v}}_{\zeta}, \widetilde{\mu})$ functions, which are called \emph{Lip-cylindrical functions}, are studied in \cite{bufetov2020spectral}. Each such function corresponds to some level $\ell\in\N$, and the collection is dense in $L^2(\mathfrak{X}^{\vect{v}}_{\zeta}, \widetilde{\mu})$. Thus, it is sufficient to show that if $f$ is a Lip-cylindrical function $f$ of level $\ell$, then the corresponding spectral measure $\sigma_f$ is singular.
The authors show that for such $f$,
\begin{equation}\label{eq local dim}
  \underline{d}(\sigma_f, \omega) \ge 2 - \frac{2\max\{0, \chi_\zeta^+(E_{S_\zeta^t}^\ell\omega\vect{v})\}}{\log\theta_1},  
\end{equation}
where $\underline{d}(\sigma_f, \omega)\coloneqq \liminf_{r\to 0}{\log\sigma_f(B_r(\omega))}/{\log r}$ is the lower local dimension of $\sigma_f$.
By \eqref{exp p ineq 2} and \eqref{eq local dim}, $\underline{d}(\sigma_f, \omega) > 1$ for Lebesgue-a.e.\ $\omega$, which implies the singularity of $\sigma_f$.

(b) By Corollary \ref{ud corollary}\ref{ud corollary a}, $({(S_\zeta^t)}^{kn} (S_\zeta^t)^{\ell}\omega \vec{1})_{n=0}^\infty$ is u.d.\ mod $1$ in $V$ for every $k\ge 1$ and $\ell\in\N$ for a.e.\ $\omega\in\R$, and it follows from Proposition \ref{exponent inequality} that
\begin{equation}\label{exp p ineq 1}
   \chi_\zeta^+(E_{S_\zeta^t}^\ell\omega\vec{1}) < \frac{\log\theta_1}{2} 
\end{equation}
for a.e.\ $\omega$.
We can now proceed as in part (a), and use the fact that the singularity of the $\R$-action associated with $\vec{1}$ implies the singularity of the $\Z$-action. Alternatively, the singularity of the $\Z$-action can be proven directly, using dimension estimates similar to \eqref{eq local dim} which hold in the $\Z$-action case; see \cite[Lemma 3.1]{bufetov2020singular} for the details
(again, the stronger assumptions made in \cite{bufetov2020singular} are used in the proof of that lemma only to obtain \eqref{exp p ineq 1}, so the lemma still holds in our case).
\end{proof}

\begin{remark} \label{remark Baake}
(a) In fact, Theorem \ref{main theorem} can be extended to $\R$-actions associated with a larger collection of vectors, but we omit the details here.

(b) Notice that given a primitive aperiodic substitution $\zeta$, we can always choose some $k\ge 1$ such that $S_\zeta^k$ is non-degenerate, and replace $\zeta$ by $\zeta^k$ without changing the substitution space. It is also not hard to remove the assumption that ${S_\zeta^t|}_V$ is non-singular: by the primary decomposition theorem, we can decompose $V$ further into a direct sum of invariant subspaces $V = V_0 \oplus V_1$ where $V_0$ is the generalized eigenspace that corresponds to the eigenvalue $0$. Let $\vect{v}$ be the projection of $\vec{1}$ onto $V_1$, then for every sufficiently large $n$, $(S_\zeta^t)^n \vec{1} = (S_\zeta^t)^n \vect{v}$, and we can look at the cocycle defined on the minimal subspace of $\vect{v}$ instead of $V$, where the restriction of $S_\zeta^t$ is guaranteed to be non-singular.

(c) In the case of the Perron-Frobenius eigenvector $\vect{u}$ of $S_\zeta^t$,
some related results were obtained by Baake et al.\ in terms of the \emph{Fourier matrix cocycle}, which is closely related to the spectral cocycle. In \cite[Fact 5.6]{baake2018spectral}, Baake, Grimm and Ma{\~n}ibo showed (using different notations) that for the Fibonacci substitution $\zeta$,\; $\chi_\zeta^+(\omega\vect{u})$ exists as a limit for a.e.\ $\omega\in\R$. Using the theory of Bohr-almost periodic functions, Baake, Frank, Grimm and Robinson gave in \cite[Lemma 6.16]{baake_2019} a bound, which is relatively similar to \eqref{exponent bound eq}, for some binary non-Pisot substitution.
Baake, G{\"a}hler and Ma{\~n}ibo extended this bound to the general case in \cite{baake2019renormalisation}, under the additional assumption that the function $\omega\mapsto \log \lVert \M(\omega\vect{u}, n)\rVert$ is Bohr-almost periodic (the authors mentioned that this assumption can be relaxed by using an extension of a theorem of Sobol, which can be found in \cite{baake_averaging}), and gave sufficient conditions for the diffraction spectrum to be singular.
\end{remark}

\section{Examples} \label{examples section}
In what follows, we consider a few examples of families of reducible non-Pisot substitutions (i.e., the characteristic polynomial of the substitution matrix is reducible over $\Q$, and the Perron-Frobenius eigenvalue is not a Pisot number), and prove they have purely singular spectrum.
We will use some of the techniques developed in \cite[Section 5.1 and Appendix]{baake2018spectral}, \cite[Corollary 9]{manibo2017lyapunov} and also used in \cite[Section 5]{bufetov2020singular}.
To ease notation, we write $z_j = e(-\xi_j) = \exp(-2\pi i \xi_j)$ for $j=0,1,2$.

\begin{example}
For every $m\ge 3$ define the substitution $\zeta_m$ by $0 \mapsto 0^m 12, \enspace 1\mapsto 1^{2m}02, \enspace 2\mapsto 0122$.
Its corresponding substitution matrix is
\[ S_{\zeta_m} = \begin{pmatrix}
    m & 1 & 1 \\
    1 & 2m & 1 \\
    1 & 1 & 2 \\
  \end{pmatrix}, \]
and a straightforward calculation shows that its eigenvalues $\theta_1, \theta_2, \theta_3$ satisfy $2m < \theta_1 < 2m + 1$, $\theta_2 = m$ and $1 < \theta_3 < 2$, so $\theta_1 \notin \Q$ and thus $\zeta_m$ is aperiodic by \cite[Theorem 4.6]{baake2013aperiodic}.
The corresponding matrix-valued function is
\[ \Mm(\vect{\xi}) = \begin{pmatrix}
    1+\dots+z_0^{m-1} & z_0^m & z_0^m z_1 \\
    z_1^{2m} & 1+\dots+z_1^{2m-1} & z_0 z_1^{2m} \\
    1 & z_0 & z_0 z_1(1 + z_2) \\
  \end{pmatrix},\]
and since $\vec{1},\; S_{\zeta_m}^t \vec{1},\; {(S_{\zeta_m}^t)}^2\vec{1}$ are linearly independent, the function
$\tM_{\zeta_m, \vec{1}}$ is just $\Mm$.
Using the Frobenius norm we have
\[ \lVert \Mm (\vect{\xi}) \rVert ^2 = \Big\lvert\frac{z_0^m - 1 }{z_0 - 1}\Big\rvert^2 + \Big\lvert\frac{z_1^{2m} - 1}{z_1 - 1}\Big\rvert^2 + \lvert z_2 + 1 \rvert^2 + 6,\]
whence
\begin{align*}
    \int_{\T^3}\log \lVert \Mm (\vect{\xi}) \rVert ^2 \,dm_3(\vect{\xi})
    =& \int_{\T^3}\log(\lvert z_0^m - 1 \rvert^2 \lvert z_1 - 1 \rvert^2 + \lvert z_1^{2m} - 1 \rvert^2 \lvert z_0 - 1 \rvert^2 \\
    & \hphantom{\int_{\T^3}\log(} + \lvert z_0 - 1 \rvert^2\lvert z_1 - 1 \rvert^2\lvert z_2 + 1 \rvert^2 + 6\lvert z_0 - 1 \rvert^2 \lvert z_1 - 1 \rvert^2)\,dm_3(\vect{\xi})\\
    - &\int_{\T^3} \log(\lvert z_0 - 1 \rvert^2 \lvert z_1 - 1 \rvert^2) \,dm_3(\vect{\xi}).
\end{align*}
Denote the two integrals on the right-hand side by $A$ and $B$ respectively. Applying Jensen's inequality and then Parseval's identity, we see that $A \le \log 40$.
Next, by Jensen's formula,
\mbox{$B = 2\int_{\T} \log(\lvert e(-t)-1 \rvert^2) \,dt = 0$.}
Therefore, using \eqref{inf identity} with $k=1$, we see that for every $m \ge 20 $,
\begin{equation*}
    \chi_{\zeta_m, \vec{1}} \le \frac{1}{2} \int_{\T^3}\log \lVert \Mm (\vect{\xi}) \rVert ^2 \,dm_3(\vect{\xi})
\le \frac{1}{2}\log40 \le \frac{1}{2}\log(2m) < \frac{1}{2}\log \theta_1,
\end{equation*}
and it follows from Theorem \ref{main theorem} that the substitution $\Z$-action has purely singular spectrum. 
\end{example}

\begin{example}
Define another family of substitutions $\sigma_m, \; m\ge 1$, by $0 \mapsto (01)^m 2, \enspace 1\mapsto 2(10)^m, \enspace 2\mapsto 1^{2m+2}$.
The eigenvalues of $S_{\sigma_m}$ satisfy $2m+1 < \theta_1 < 2m+2$, $-2 < \theta_2 < -1$ and $\theta_3 = 0$, and again this substitution is aperiodic. Denote $q(z_0, z_1) = 1+(z_0 z_1)+\dots+(z_0 z_1)^{m-1} = \frac{(z_0 z_1)^m - 1}{z_0 z_1 - 1}$, then we have
\[ \Msm(\vect{\xi}) = \begin{pmatrix}
    q(z_0, z_1) & z_0 q(z_0, z_1) & (z_0 z_1)^m\\
    z_1 z_2 q(z_0, z_1) & z_2 q(z_0, z_1) & 1 \\
    0 & 1+z_1+\dots+z_1^{2m+1} & 0 \\
  \end{pmatrix}.\]
The minimal subspace of both the Perron-Frobenius eigenvector $\vect{u}$ and $\vec{1}$ (with respect to $S_{\sigma_m}^t$) is $V = \Span\{(1,1,0)^t, \, (0,0,1)^t\}$ (notice that with $m=1$ and $\vect{u}$, this is Example \ref{minimal subspace example}(b) and also Example \ref{matrix function example}), and hence $\tM_{\sigma_m, \vect{u}} = \tM_{\sigma_m, \vec{1}} = \tM_{\sigma_m, \vect{w}}$,
where $\vect{w}$ is any positive vector in $V$ (the last equality follows from Example \ref{minimal subspace example}(c)).
When restricted to $V$, $\lVert \Msm (\vect{\xi}) \rVert ^2$ is simplified into
\[ \lVert \tM_{\sigma_m, \vect{w}} (x_0, x_1) \rVert ^2 = \lVert \Msm (x_0, x_0, x_1) \rVert ^2 = 4\Big\lvert\frac{z_0^{2m} - 1 }{z_0^2 - 1}\Big\rvert^2 + \Big\lvert\frac{z_0^{2m+2} - 1 }{z_0 - 1}\Big\rvert^2 + 2.\]
(where this time $z_0 = e(-x_0)$).
Consequently,
\begin{align*}
    & \int_{\T^2}\log\lVert \tM_{\sigma_m, \vect{w}} (\vect{x}) \rVert ^2 \,dm_2(\vect{x}) \\
    = &\int_{\T}\log(4\lvert z_0^{2m} - 1\rvert^2 + \lvert z_0^{2m+2} - 1 \rvert^2 \lvert z_0 + 1\rvert^2 + 2\lvert z_0^2-1 \rvert^2) \,dx_0
    - \int_{\T} \log(\lvert z_0^2-1 \rvert^2) \,dx_0.
\end{align*}
Proceeding as in the previous example, for every $m\ge 8$ we have
\[ \chi_{\sigma_m,\vect{w}}\le \frac{1}{2} \log(16) < \frac{1}{2} \log(\theta_1),\]
so by Theorem \ref{main theorem}, the $\Z$-action and any $\R$-action associated with a positive vector in $V$ have purely singular spectrum. Moreover, \cite[Corollary 4.5]{solomyak1997dynamics} immediately implies that the self-similar action associated with $\vect{u}$ is singular continuous.
\end{example}

\begin{example}
Define $\zeta\coloneqq\zeta_{m,A,B}$ by $0 \mapsto A2, \enspace 1\mapsto 2B, \enspace 2\mapsto 022$, where $A, B\in \{0,1\}^m$.
Suppose that $A\neq 0^m$ and that in each of the words $A, B$, its less frequent letter appears at most $k$ times, where $8k^2 + 8k + 14 \le m$. The eigenvalues of $S_\zeta$ satisfy $m < \theta_1 < m+1$, $\theta_2 = \ell_0(A) - \ell_0(B)$ and $1<\theta_3<2$, where $\ell_0(A)$ and $\ell_0(B)$ are the number of $0$'s in $A$ and $B$ respectively.
The minimal subspace of both $\vec{1}$ and the Perron-Frobenius eigenvector $\vect{u}$ is again $\Span\{(1,1,0)^t, \, (0,0,1)^t\}$. Using the notation $z_j = e(-x_j)$ we get
\begin{align*}
    & \int_{\T^2}\log\lVert \tM_{\zeta, \vec{1}} (\vect{x}) \rVert ^2 \,dm_2(\vect{x})
    = \int_{\T^2}\log(3+ \lvert 1+z_1\rvert^2 + \sum_{b,c=0,1}\lvert (\M(x_0,x_0,x_1))_{bc} \rvert^2) \,dm_2(\vect{x})\\
    \le &\int_{\T^2} \log(2(\lvert 1 + \dots + z_0^{m-1} \rvert + k)^2 + 2k^2 + 3 + \lvert 1+z_1\rvert^2) \,dm_2(\vect{x})
    \\ =& \int_{\T^2} \log(2\lvert z_0^m - 1\rvert^2 + 4k\lvert z_0^m - 1\rvert \lvert z_0 - 1 \rvert +(4k^2 + 3)\lvert z_0-1\rvert^2 + \lvert 1+z_1\rvert^2\lvert z_0-1\rvert^2) \,dm_2(\vect{x}),
\end{align*}
and it follows from Jensen inequality, Parseval's identity and Cauchy–Schwarz inequality that
\[ \chi_{\zeta, \vect{u}} = \chi_{\zeta, \vec{1}} \le \frac{1}{2} \log(8k^2+8k+14) < \frac{1}{2} \log(\theta_1),\]
and both associated actions, as well as any other $\R$-action associated with a positive vector in this subspace, are purely singular.
\end{example}
\bigskip
\textbf{Acknowledgements:} The author is grateful to Boris Solomyak for many helpful ideas, suggestions and comments.
This research is a part of the author's master's thesis (in preparation) at the Bar-Ilan University under the direction of B. Solomyak and was supported in part by the Israel Science Foundation grant 911/19 (PI B. Solomyak).
\bigskip
\bibliographystyle{plain}
\bibliography{references}
\end{document}